\documentclass[11pt]{article}

\usepackage{latexsym,color,graphicx,amsmath,amssymb,amsthm,graphics}
\usepackage{subfig}

\def\reals{{\mathbb R}}
\def\R{{\mathbb R}}
\def\P{{\mathbb P}}

\def\N{{\mathbb N}}
\def\eps{{\varepsilon}}
\def\vphi{\varphi}
\def\si{\sigma}

\newtheorem{theorem}{Theorem}
\newtheorem{lemma}[theorem]{Lemma}
\newtheorem{prop}[theorem]{Proposition}

\theoremstyle{definition}
\newtheorem{defn}[theorem]{Definition}

\begin{document}

\title{The number of unit-area triangles in the plane: Theme and variations\footnote{Work on this paper by Orit E. Raz and Micha Sharir was supported by Grant 892/13 from the Israel Science Foundation and by the Israeli Centers of Research Excellence (I-CORE) program (Center No.~4/11). Work by Micha Sharir was also supported by Grant 2012/229 from the U.S.--Israel Binational Science Foundation and by the Hermann Minkowski-MINERVA Center for Geometry at Tel Aviv University.
Part of this research was performed while the authors were visiting the Institute for Pure and Applied Mathematics (IPAM), which is supported by the National Science Foundation. 
}}

\author{
Orit E. Raz\thanks{%
School of Computer Science, Tel Aviv University,
Tel Aviv 69978, Israel.
{\sl oritraz@post.tau.ac.il} }
\and
Micha Sharir\thanks{%
School of Computer Science, Tel Aviv University,
Tel Aviv 69978, Israel.
{\sl michas@post.tau.ac.il} }
}
\maketitle

\begin{abstract}
We show that the number of unit-area triangles determined by a set $S$ of $n$ points in the plane is $O(n^{20/9})$, improving the earlier bound $O(n^{9/4})$ of Apfelbaum and Sharir~\cite{AS10}. We also consider two special cases of this problem: (i) We show, using a somewhat subtle construction, that if $S$ consists of points on three lines, the number of unit-area triangles that $S$ spans can be $\Omega(n^2)$, for any triple of lines (it is always $O(n^2)$ in this case). (ii) We show that if $S$ is a {\em convex grid} of the form $A\times B$, where $A$, $B$ are {\em convex} sets of $n^{1/2}$ real numbers each (i.e., the sequences of differences of consecutive elements of $A$ and of $B$ are both strictly increasing), then $S$ determines $O(n^{31/14})$ unit-area triangles.
\end{abstract}


\section{Introduction}

In 1967, Oppenheim (see \cite{EP95}) asked the following question:
Given $n$ points in the plane and $A>0$, how many triangles spanned
by the points can have area $A$? By applying a scaling transformation,
one may assume $A=1$ and count the triangles of {\em unit} area.
Erd\H{o}s and Purdy~\cite{EP71} showed that a $\sqrt{\log n}\times
(n/\sqrt{\log n})$ section of the integer lattice determines
$\Omega(n^2 \log\log{n})$ triangles of the same area. They also showed
that the maximum number of such triangles is at most $O(n^{5/2})$. In
1992, Pach and Sharir~\cite{PS92} improved the bound to
$O(n^{7/3})$, using the Szemer\'edi-Trotter theorem~\cite{ST83} (see below)
on the number of point-line incidences. More recently,
Dumitrescu et al.~\cite{DST08} have further improved the upper bound 
to $O(n^{44/19})=O(n^{2.3158})$, by estimating the number of incidences
between the given points and a 4-parameter family of quadratic curves. 
In a subsequent improvement, Apfelbaum and Sharir~\cite{AS10} have obtained
the upper bound $O(n^{9/4+\eps})$, for any $\eps>0$, which has been 
slightly improved to $O(n^{9/4})$ in Apfelbaum~\cite{Apf-thesis}. This has been 
the best known upper bound so far.

In this paper we further improve the bound to $O(n^{20/9})$.
Our proof uses a different reduction of the problem to an incidence
problem, this time to incidences between points and two-dimensional 
algebraic surfaces in $\reals^4$. A very recent result of Solymosi and De Zeeuw~\cite{SodZ}
provides a sharp upper bound for the number of such incidences, similar
to the Szemer\'edi--Trotter bound, provided that the points, surfaces,
and incidences satisfy certain fairly restrictive assumptions. The main 
novel features of our analysis are thus (a) the reduction of the
problem to this specific type of incidence counting, and (b) showing
that the assumptions of \cite{SodZ} are satisfied in our context.

After establishing this main result, we consider two variations, in which
better bounds can be obtained.

We first consider the case where the input points lie on three 
arbitrary lines. It is easily checked that in this case there are at most 
$O(n^2)$ unit-area triangles. We show, in Section~\ref{sec:exa},
that this bound is tight, and can be attained for any triple 
of lines.
Rather than just presenting the construction, we spend some time
showing its connection to a more general problem studied by Elekes
and R\'onyai~\cite{ER00} (see also the recent developments in 
\cite{ESz12,RSS1,RSdZ}), involving the zero set of a trivariate polynomial 
within a triple Cartesian product. Skipping over the details, which are 
spelled out in Section~\ref{sec:exa}, it turns out that the case of
unit-area triangles determined by points lying on three lines is an
exceptional case in the theory of Elekes and R\'onyai~\cite{ER00}, which
then leads to a construction with $\Theta(n^2)$ unit-area triangles.

Another variation that we consider concerns unit-area triangles spanned by points in a {\em convex grid}.
That is, the input set is of the form $A\times B$, where $A$ and $B$
are convex sets  of $n^{1/2}$ real numbers each; a set of real numbers
is called {\em convex} if the differences between consecutive elements form
a strictly increasing sequence. We show that in this case $A\times B$ 
determine $O(n^{31/14})$ unit-area triangles. The main technical 
tool used in our analysis is a result of Schoen and Shkredov~\cite{SS11}
on difference sets involving convex sets.\footnote{Very recently, in work in progress, jointly with I. Shkredov, 
the bound is further improved in this case.}


\section{Unit-area triangles in the plane}\label{sec:general}
\begin{theorem}\label{main1}
The number of unit-area triangles spanned by $n$ points in the plane is $O(n^{20/9})$.
\end{theorem}
We first recall the Szemer\'edi--Trotter theorem \cite{ST83} on point-line incidences in the plane.
\begin{theorem}[{\bf Szemer\'edi and Trotter~\cite{ST83}}]\label{ST}
(i) The number of incidences between $M$ distinct points and $N$ distinct lines in the plane is 
$O(M^{2/3}N^{2/3}+M+N)$.
(ii) Given $M$ distinct points in the plane and a parameter $k\le M$, the number of lines incident to at least 
$k$ of the points is $O(M^2/k^3+M/k)$. Both bounds are tight in the worst case.
\end{theorem}
\noindent{\bf Proof of Theorem~\ref{main1}.} 
Let $S$ be a set of $n$ points in the plane, and let $U$ denote the set of unit-area triangles spanned by $S$. For any pair of distinct points, $p\neq q \in S$, let $\ell_{pq}$
denote the line through $p$ and $q$. The points $r$ for which the triangle $pqr$ has unit area lie on two lines $\ell_{pq}^-,\ell_{pq}^+$ parallel to $\ell_{pq}$ and at distance $2/|pq|$ from $\ell_{pq}$ on either side. We let
$\ell_{pq}'\in\{\ell_{pq}^-,\ell_{pq}^+\}$ be the line that lies to the left of the vector $\vec{pq}$. 
We then have
$$
|U|=\frac13 \sum_{(p,q)\in S\times S,\; p\neq q}|\ell_{pq}'\cap S|.
$$

It suffices to consider only triangles $pqr$ of $U$, that have the property that at least one of the three lines $\ell_{pq}$, $\ell_{pr}$, $\ell_{qr}$ is incident to at most $n^{1/2}$ points of $S$, because the number of triangles in $U$ that do not have this property is $O(n^{3/2})$. 
Indeed, by Theorem~\ref{ST}(ii), 
there exist at most $O(n^{1/2})$ lines in $\R^2$, such that each contains at least $n^{1/2}$ points of $S$. 
Since every triple of those lines supports (the edges of) at most one triangle (some of the lines might be mutually parallel, 
and some triples might intersect at points that do not belong to $S$), 
these lines support in total at most $O(n^{3/2})$ triangles, and, in particular, at most $O(n^{3/2})$ triangles of $U$. 
Since this number is subsumed in the asserted bound on $|U|$, we can therefore ignore such triangles in our analysis. 
In what follows, $U$ denotes the set of the remaining unit-area triangles.

We charge each of the surviving unit-area triangles $pqr$ to one of its sides, say $pq$, 
such that $\ell_{pq}$ contains at most $n^{1/2}$ points of $S$. 
That is, we have 
$$
|U|\le \sum_{(p,q)\in (S\times S)^*}|\ell_{pq}'\cap S|,
$$
where $(S\times S)^*$ denotes the subset of pairs $(p,q)\in S\times S$, such that $p\neq q$, and the line $\ell_{pq}$ is incident to at most $n^{1/2}$ points of $S$.

A major problem in estimating $|U|$ is that the lines $\ell_{pq}'$, for $p,q\in S$, are not necessarily distinct, 
and the analysis has to take into account the (possibly large) multiplicity of these lines. 
(If the lines were distinct then $|U|$ would be bounded by the number of incidences between $n(n-1)$ lines and $n$ points, 
which is $O(n^2)$ --- see Theorem~\ref{ST}(i).)
Let $L$ denote the collection of lines $\{\ell_{pq}'\mid (p,q)\in (S\times S)^*\}$ (without multiplicity). 
For $\ell\in L$, we define $(S\times S)_{\ell}$ to be the set of all pairs $(p,q)\in S\times S$, with $p\neq q$, and for which $\ell_{pq}'=\ell$. 
We then have
$$
|U|
\le  \sum_{\ell\in L}|\ell\cap S||(S\times S)_\ell|.
$$

Fix some integer parameter $k\le n^{1/2}$, to be set later, and partition $L$ into the sets 
\begin{align*}
&L^-=\{\ell\in L\mid |\ell\cap S|< k\},\\
&L^+=\{\ell\in L\mid k\le |\ell\cap S|\le n/k\},\\
&L^{++}=\{\ell\in L\mid |\ell\cap S|> n/k\}.
\end{align*}
We have
$$
|U| 
\le \sum_{\ell\in L^-}|\ell\cap S||(S\times S)_\ell|+ \sum_{\ell\in L^+}|\ell\cap S||(S\times S)_\ell|+\sum_{\ell\in L^{++}}|\ell\cap S||(S\times S)_\ell|.
$$
The first sum is at most $
k\sum_{\ell\in L^-}|(S\times S)_\ell|\le kn^2$,
because $\sum_{\ell\in L^-}|(S\times S)_\ell|$ is at most $|(S\times S)^*|\le |S\times S|=n^2$.
The same (asymptotic) bound also holds for the the third sum. 
Indeed, since $n/k\ge n^{1/2}$, the number of lines in $L^{++}$ is at most $O(k)$, 
as follows from Theorem~\ref{ST}(ii),
and, for each $\ell\in L^{++}$, we have $|\ell\cap S|\le n$ and $|(S\times S)_\ell|\le n$ (for any $p\in S$, $\ell\in L$, 
there exists at most one point $q\in S$, such that $\ell_{pq}'=\ell$). This yields a total of at most $O(n^2 k)$ unit-area triangles.
It therefore remains to bound the second sum, over $L^+$.

Applying the Cauchy-Schwarz inequality to the second sum, it follows that
$$
|U|\le O(n^2k)+\left(\sum_{\ell\in L^+}|\ell\cap S|^2\right)^{1/2}\left(\sum_{\ell\in L^+}|(S\times S)_\ell|^2\right)^{1/2}.
$$

Let $N_j$ (resp., $N_{\ge j}$), for $k\le j\le n/k$, denote the number of lines $\ell\in L^+$ for which $|\ell\cap S|=j$ (resp., $|\ell\cap S|\ge j$). 
By Theorem~\ref{ST}(ii), $\displaystyle{N_{\ge j}=O\left({n^2}/{j^3}+{n}/{j}\right)}$. Hence 
\begin{align*}
\sum_{\ell\in L^+}|\ell\cap S|^2
& =\sum_{j= k}^{n/k}j^2N_j
\le k^2N_{\ge k}+\sum_{j=k+1}^{n/k}(2j-1)N_{\ge j} \\
& =O\left(\frac{n^2}{k}+nk+\sum_{j=k+1}^{n/k}\left(\frac{n^2}{j^2}+n\right)\right)
=O\left(\frac{n^2}{k}\right)
\end{align*}
(where we used the fact that $k\le n^{1/2}$).
It follows that
$$
|U|=O\left(n^2k+\frac{n}{k^{1/2}}\Big(\sum_{\ell\in L^+} |(S\times S)_\ell|^2\Big)^{1/2}\right).
$$
To estimate the remaining sum, put
$$
Q:=\left\{(p,u,q,v)\in S^4\mid (p,u),(q,v)\in (S\times S)_\ell,~\text{for some}~\ell\in L^+\right\}.
$$
That is, $Q$ consists of all quadruples $(p,u,q,v)$ such that $\ell_{pu}'=\ell_{qv}'\in L^+$, and each of $\ell_{pu},\ell_{qv}$ contains 
at most $n^{1/2}$ points of $S$. 
See Figure~\ref{figs}(a) for an illustration. 
\begin{figure}%
\centering
\subfloat[][]{\includegraphics[width=0.4\textwidth]{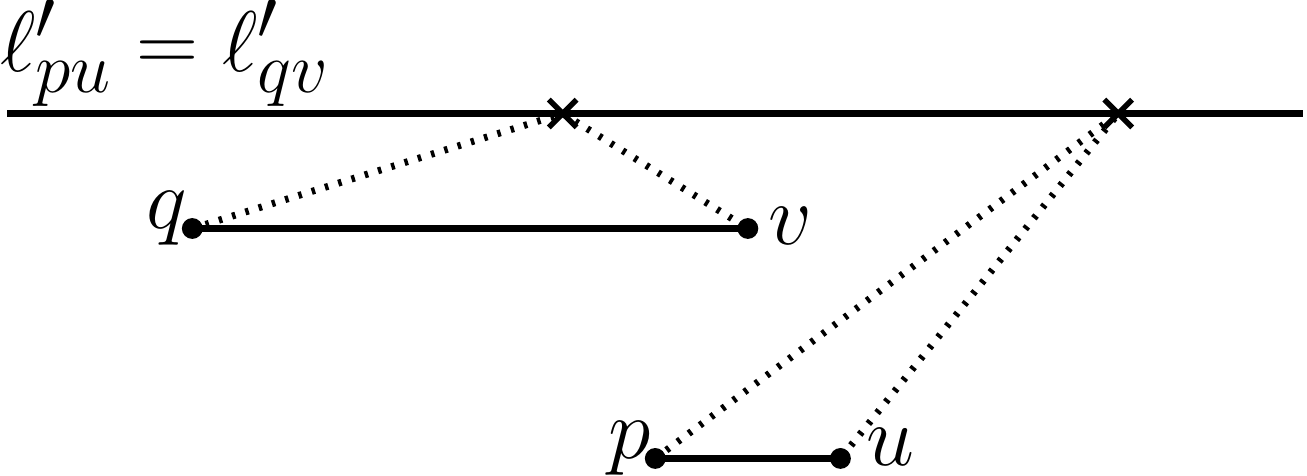}}%
\qquad
\subfloat[][]{\includegraphics[width=0.4\textwidth]{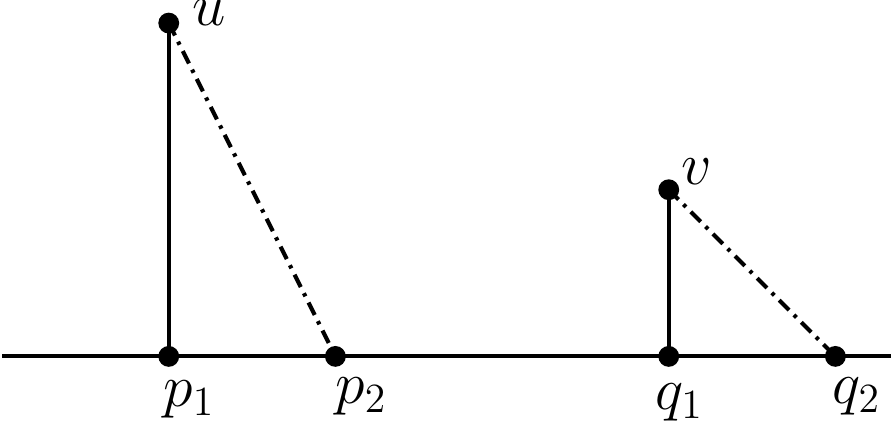}}
\caption{(a) A quadruple $(p,u,q,v)$ in $Q$. (b) If $p_1, q_1,p_2,q_2$ are collinear and $|p_1p_2|=|q_1q_2|$ then $\ell_{p_2u},\ell_{q_2v}$ are not parallel to one another, for every $(u,v)\in\sigma_{p_1q_1}\setminus\ell_{p_1q_1}$. Thus, in particular, 
$(u,v)\not\in\sigma_{p_2q_2}$.}%
\label{figs}%
\end{figure}
The above bound on $|U|$ can then be written as
\begin{equation}\label{eq:UQ}
|U|=O\left( n^2k+\frac{n|Q|^{1/2}}{k^{1/2}}\right).
\end{equation}

The main step of the analysis is to establish the following upper bound on $|Q|$.
\begin{prop}\label{prop:Q}
Let $Q$ be as above. Then
$
|Q|=O\left( n^{8/3}\right).
$
\end{prop}
The proposition, combined with (\ref{eq:UQ}), implies that 
$|U|=O\left( n^2k+{n^{7/3}}/{k^{1/2}}\right)$,
which, if we choose $k=n^{2/9}$, becomes $|U|=O(n^{20/9})$.
Since the number of triangles that we have discarded is only $O(n^{3/2})$, Theorem~\ref{main1} follows.
\qed\\

\noindent{\bf Proof of Proposition~\ref{prop:Q}.}
Consider first quadruples $(p,u,q,v)\in Q$, with all four points $p,u,q,v$ collinear. 
As is easily checked, in this case $(p,u,q,v)$ must also satisfy $|pu|=|qv|$. It follows that a line $\ell$ in the plane, 
which is incident to at most $j$ points of $S$, can support at most $j^3$ such quadruples. 
By definition, $(S\times S)_\ell\subset (S\times S)^*$ for each $\ell\in L^+$,
so the line $\ell_{pu}=\ell_{qv}$ is incident to at most $n^{1/2}$ points of $S$, 
and it suffices to consider only lines $\ell$ with this property.
Using the preceding notations $N_j$, $N_{\ge j}$, the number of quadruples under consideration is
$$
O\left(\sum_{j\le n^{1/2}}j^3N_j\right)
=O\left(\sum_{j\le n^{1/2}}j^2N_{\ge j}\right)
=O\left(\sum_{j\le n^{1/2}}j^2\cdot\frac{n^2}{j^3}\right)
=O\left(n^2\log n\right).
$$
This is subsumed by the asserted bound on $|Q|$, so, in what follows 
we only consider quadruples $(p,u,q,v)\in Q$, such that $p,u,q,v$ are not collinear.

For convenience, we assume that no pair of points of $S$ share the same $x$- or $y$-coordinate; this can always be enforced by a suitable rotation of the coordinate frame. 
The property that two pairs of $S\times S$ are associated with a common line of $L$ can then be expressed in the following algebraic manner.
\begin{lemma}\label{lem:surfdef}
Let $(p,u,q,v)\in S^4$, and represent $p=(a,b), u=(x,y), q=(c,d)$, and $v=(z,w)$, by their coordinates in $\R^2$. Then $\ell_{pu}'=\ell_{qv}'$ if and only if
\begin{equation}\label{surfdef}
\frac{y-b}{x-a}=\frac{w-d}{z-c}\quad\quad\text{and}\quad\quad
\frac{bx-ay+2}{x-a}=\frac{dz-cw+2}{z-c}.
\end{equation}
\end{lemma}
\noindent{\bf Proof.}
Let $\alpha,\beta\in\R$ be such that $\ell_{(a,b)(x,y)}'=\{(t,\alpha t+\beta)\mid t\in\R\}$. Then, by the definition of $\ell_{(a,b)(x,y)}'$, we have 
$$
\frac12\left|
\begin{array}{ccc}
a& x& t\\
b& y& \alpha t +\beta\\
1&1&1
\end{array}\right|=1,
$$
or
$$
(b-y-\alpha(a-x))t-\beta(a-x)+ay-bx=2,
$$
for all $t\in\R$.
Thus,
\begin{align*}
\alpha&=\alpha(a,b,x,y)=\frac{y-b}{x-a},\\
\beta&=\beta(a,b,x,y)=\frac{bx-ay+2}{x-a}.
\end{align*}
Then the constraint $\displaystyle{\ell_{(a,b)(x,y)}'\equiv\ell_{(c,d)(z,w)}'}$ can be written as 
\begin{align*} 
\alpha(a,b,x,y)&=\alpha(c,d,z,w),\\
\beta(a,b,x,y)&=\beta(c,d,z,w),
\end{align*}
which is \eqref{surfdef}.
\qed
\medskip

We next transform the problem of estimating $|Q|$ into an incidence problem. 
With each pair $(p=(a,b),q=(c,d))\in (S\times S)^*$, 
we associate the two-dimensional surface $\sigma_{pq}\subset\R^4$ which is the locus of 
all points $(x,y,z,w)\in\R^4$ that satisfy the system (\ref{surfdef}). 
The degree of $\sigma_{pq}$ is at most $4$, being the intersection of two quadratic hypersurfaces.
We let $\Sigma$ denote the set of surfaces 
$$
\Sigma:=\{\sigma_{pq}\mid (p,q)\in (S\times S)^*\}.
$$
For $(p_1,q_1)\neq (p_2,q_2)$, the corresponding surfaces $\sigma_{p_1q_1},\sigma_{p_2q_2}$ are distinct.
The proof of this fact is not difficult, but is somewhat cumbersome, and we therefore omit it,
since our analysis does not use this property.
We also consider the set
$
\Pi:=(S\times S)^*, 
$
regarded as a point set in $\R^4$ (identifying $\R^2\times \R^2\simeq \R^4$).
We have $|\Pi|=|\Sigma|=O(n^2)$. The set $I(\Pi,\Sigma)$, the set of {\em incidences} between $\Pi$ and $\Sigma$, 
is naturally defined as
$$
I(\Pi,\Sigma):=\{ (\pi,\sigma)\in\Pi\times\Sigma \mid \pi\in\sigma \}.
$$
By Lemma \ref{lem:surfdef}, we have
$ (x,y,z,w)\in\sigma_{pq}~\text{if and only if}~\ell_{pu}'=\ell_{qv}'$,
where $u:=(x,y)$ and $v:=(z,w)$. This implies that $|Q|\le |I(\Pi,\Sigma)|$. 

Consider the subcollection $\mathcal I$ of incidences $((x,y,z,w),\si_{pq})\in I(\Pi,\Sigma)$, 
such that $p,q,u:=(x,y),v:=(z,w)$ are non-collinear (as points in $\R^2$). As already argued, the number
of collinear quadruples  in $Q$ is $O(n^2\log n)$, and hence
$|Q|\le |\mathcal I|+O(n^2\log n)$. So to bound $|Q|$ it suffices to obtain an upper bound on $|\mathcal I|$.

For this we use the following recent result of Solymosi and De Zeeuw~\cite{SodZ} (see also the related results in \cite{SoT12,Zah}).
To state it we need the following definition, which is a specialized version of the more general original definition
in~\cite{SodZ}. 
\begin{defn}\label{def:good}
A two-dimensional constant-degree surface $\sigma$ in $\R^4$ is said to be {\em slanted} (the original term used in~\cite{SodZ} is {\em good}), if, 
for every $p\in \R^2$, $\rho_i^{-1}(p)\cap \sigma$ is finite, for $i=1,2$, where $\rho_1$ and $\rho_2$ are the projections of $\R^4$ 
onto its first and last pairs of coordinates, respectively.
\end{defn}

\begin{theorem}[{\bf Solymosi and De Zeeuw~\cite{SodZ}}]\label{SodZ}
Let $S$ be a subset of $\R^2$, and let $\Gamma$ be a finite set of two-dimensional constant-degree slanted surfaces.
Set $\Pi:=S\times S$, and let $\mathcal I\subset I(\Pi,\Gamma)$. Assume that 
for every pair of distinct points $\pi_1,\pi_2\in \Pi$ there are at most $O(1)$ 
surfaces $\sigma\in\Sigma$ such that both pairs
$(\pi_1,\sigma),(\pi_2,\sigma)$ are in $\mathcal I$. Then 
$$
|\mathcal I|=O\left(|\Pi|^{2/3}|\Sigma|^{2/3}+|\Pi|+|\Sigma|\right).
$$
\end{theorem}

To apply Theorem~\ref{SodZ}, we need the following key technical proposition, whose proof is given in the next subsection.
\begin{prop}\label{prop:admis}
Let $\Pi$, $\Sigma$, and $\mathcal I$ be the sets that arise in our setting, as specified above. 
Then, (a) the surfaces of $\Sigma$ are all slanted, and (b) for every pair of distinct points $\pi_1,\pi_2\in \Pi$, there are at most three surfaces $\sigma\in\Sigma$ such that both pairs
$(\pi_1,\sigma),(\pi_2,\sigma)$ are in $\mathcal I$.
\end{prop}

We have $|\Pi|, |\Sigma|=O(n^2)$. Therefore, Theorem~\ref{SodZ} implies that  $|\mathcal I|=O( n^{8/3})$, which completes the proof of Proposition~\ref{prop:Q}
(and, consequently, of Theorem~\ref{main1}).
\qed

\subsection{Proof of Proposition~\ref{prop:admis}}\label{sec:prfadmis}
We start by eliminating $z$ and $w$ from (\ref{surfdef}). An
easy calculation shows that
\begin{align}\label{eqxy}
z & = \frac{2(x-a)}{(b-d)(x-a)+(c-a)(y-b)+2}+c,\\
w & = \frac{2(y-b)}{(b-d)(x-a)+(c-a)(y-b)+2} +d. \nonumber
\end{align}
This expresses $\sigma_{pq}$ as the graph of a linear rational function from $\reals^2$ to $\reals^2$
(which is undefined on the line at which the denominator vanishes). Passing to homogeneous coordinates,
replacing $(x,y)$ by $(x_0,x_1,x_2)$ and $(z,w)$ by $(z_0,z_1,z_2)$, we can re-interpret $\sigma_{pq}$ 
as the graph of a projective transformation $T_{pq}:\P\R^2\to\P\R^2$, given by
$$
\left(\begin{array}{c}
z_0 \\ z_1 \\z_2
\end{array}\right) = 
\left(\begin{array}{ccc}
ad-bc+2 & b-d & c-a \\
c(ad-bc)+2(c-a) & c(b-d)+2 & c(c-a) \\
d(ad-bc)+2(d-b) & d(b-d) & d(c-a)+2 
\end{array}\right)
\left(\begin{array}{c}
x_0 \\ x_1 \\x_2
\end{array}\right) .
$$

The representation \eqref{eqxy} implies that every $(x,y)$ defines at most one pair $(z,w)$ such that 
$(x,y,z,w)\in\sigma_{pq}$. By the symmetry of the definition of $\sigma_{pq}$, every pair $(z,w)$ also determines at most one pair $(x,y)$ such that $(x,y,z,w)\in\sigma_{pq}$.
This shows that, for any $p\ne q\in \R^2$, the surface $\sigma_{pq}$ is slanted, which proves Proposition~\ref{prop:admis}(a).

For Proposition~\ref{prop:admis}(b), it is equivalent, by the symmetry of the setup, to prove the following dual statement:
For any $p_1\neq q_1,p_2\neq q_2\in S$, such that $(p_1,q_1)\neq (p_2,q_2)$, we have $|\sigma_{p_1q_1}\cap\sigma_{p_2q_2}\cap\mathcal I|\le 3$.

Let $p_1,q_1,p_2,q_2\in S$ be as above, and assume that $|\sigma_{p_1q_1}\cap\sigma_{p_2q_2}\cap\mathcal I|\ge 4$.
Note that this means that the two projective transformations $T_{p_1q_1}$, $T_{p_2q_2}$
agree in at least four distinct points of the projective plane.
We claim that in this case $\sigma_{p_1q_1}$ and $\sigma_{p_2q_2}$, regarded as graphs of functions
on the affine $xy$-plane, must coincide on some line in that plane.

This is certainly the case if $\sigma_{p_1q_1}$ and $\sigma_{p_2q_2}$ coincide. (As mentioned
earlier, this situation cannot arise, but we include it since we did not provide a proof of its
impossibility.) 
We may thus assume that these surfaces are distinct, which implies that $T_{p_1q_1}$ and $T_{p_2q_2}$
are distinct projective transformations.

As is well known, two distinct projective transformations of the plane cannot agree at four distinct 
points so that no three of them are collinear. Hence, out of the four points at which
$T_{p_1q_1}$ and $T_{p_2q_2}$ agree, three must be collinear.
Denote this triple of points (in the projective $xy$-plane) as $u_1,u_2,u_3$, and their respective images
(in the projective $zw$-plane) as $v_i = T_{p_1q_1}(u_i)=T_{p_2q_2}(u_i)$, for $i=1,2,3$.
Then the line $\lambda$ that contains $u_1$, $u_2$, $u_3$ is mapped by both
$T_{p_1q_1}$ and $T_{p_2q_2}$ to a line $\lambda^*$, and both transformations
coincide on $\lambda$ (since they both map the three distinct points $u_1$, $u_2$, $u_3$ to the
same three respective points $v_1$, $v_2$, $v_3$).

Passing back to the affine setting, let then $\lambda,\lambda^*$ be a pair of lines in the 
$xy$-plane and the $zw$-plane, respectively, such that,
for every $(x,y)\in\lambda$ (other than the point at which the denominator in (\ref{eqxy}) vanishes)
there exists $(z,w)\in\lambda^*$, satisfying
$(x,y,z,w)\in\sigma_{p_1q_1}\cap\sigma_{p_2q_2}$.
We show that in this case $p_1,q_1,p_2,q_2$ are all collinear and $|p_1p_2|=|q_1q_2|$.

We first observe that  $\ell_{p_1p_2}\parallel \ell_{q_1q_2}$. 
Indeed, if each of $\lambda\cap \ell_{p_1p_2}$ and $\lambda\cap \ell_{q_1q_2}$ is either empty or infinite, then we must have
$\ell_{p_1p_2}\parallel\ell_{q_1q_2}$ (since both are parallel to $\lambda$).
Otherwise, assume without loss of generality that $|\ell_{p_1p_2}\cap \lambda|=1$, and let $\xi$ denote the unique point in this intersection.
Let $\eta$ be the point such that $(\xi,\eta)$ satisfies \eqref{eqxy} with respect 
to both surfaces $\sigma_{p_1q_1}$, $\sigma_{p_2q_2}$ (the same point arises for both surfaces because $\xi\in\lambda$).
That is, $\ell_{p_1\xi}'= \ell_{q_1\eta}'$, and $\ell_{p_2\xi}'= \ell_{q_2\eta}'$.
In particular,  $\ell_{p_1\xi}\parallel\ell_{q_1\eta}$, and $\ell_{p_2\xi}\parallel\ell_{q_2\eta}$.
Since, by construction, $\xi\in\ell_{p_1p_2}$, we have $\ell_{p_1\xi}\equiv\ell_{p_2\xi}$, which yields that also
$\ell_{q_1\eta}\parallel\ell_{q_2\eta}$. Thus necessarily $q_1,q_2,\eta$ are collinear, and 
$\ell_{q_1q_2}\parallel\ell_{p_1p_2}$, as claimed.

Assume that at least one of $\ell_{p_1q_1}$, $\ell_{p_2q_2}$ intersects $\lambda$ in exactly one point; say, without loss of generality, it is $\ell_{p_1q_1}$, and let $\xi$ denote the unique point in this intersection. Similar to the argument just made, let $\eta$ be the point such that $(\xi,\eta)$ satisfies \eqref{eqxy} with respect 
to both surfaces $\sigma_{p_1q_1}$, $\sigma_{p_2q_2}$.
Note that since $\xi\in\ell_{p_1q_1}$, we must have $\eta\in\ell_{p_1q_1}$ too, and $|p_1\xi|=|q_1\eta|$.
In particular, since $p_1\neq q_2$, by assumption, we also have $\xi\neq \eta$.
Using the properties $\ell_{p_1p_2}\parallel\ell_{q_1q_2}$ and $(\xi,\eta)\in\sigma_{p_2q_2}$, it follows that the triangles 
$p_1\xi p_2$, $q_1\eta q_2$ are congruent; see Figure~\ref{lemfig}(a). Thus, in particular, $|p_2\xi|=|q_2\eta|$.
Since, by construction, also $\ell_{p_2\xi}'\equiv\ell_{q_2\eta}'$, it follows that $p_2,q_2\in\ell_{\xi\eta}$.
We conclude that in this case $p_1,q_1,p_2,q_2$ are collinear and $|p_1p_2|=|q_1q_2|$.

\begin{figure}%
\centering
\subfloat[][]{\includegraphics[width=0.28\textwidth]{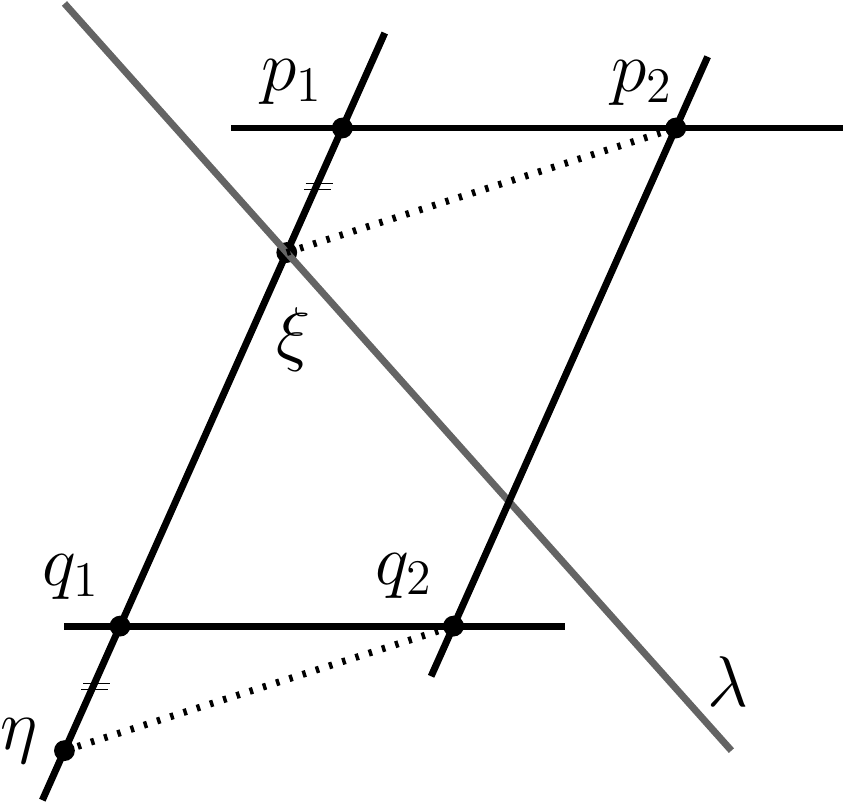}}%
\qquad\qquad\qquad\qquad
\subfloat[][]{\includegraphics[width=0.3\textwidth]{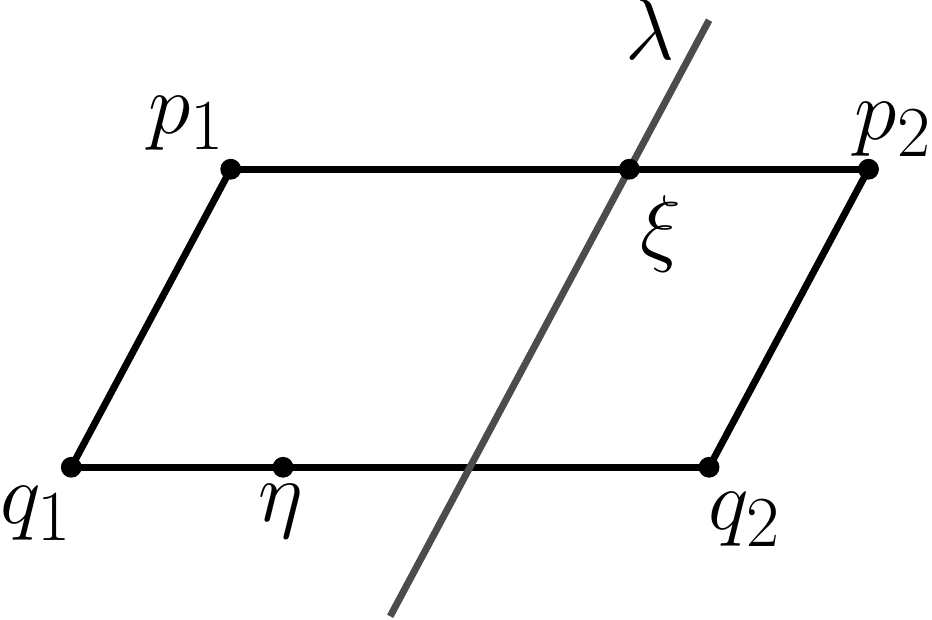}}
\caption{(a) The properties $|p_1\xi|=|q_1\eta|$, $\ell_{p_1p_2}\parallel\ell_{q_1q_2}$, and $(\xi,\eta)\in\sigma_{p_2q_2}$ imply that the triangles 
$p_1\xi p_2$, $q_1\eta q_2$ are congruent, and therefore $\ell_{p_1q_1}$, $\ell_{p_2q_2}$ must be parallel to one another.  (b) $p_1q_1q_2p_2$ is a parallelogram and $\lambda$ is parallel to $\ell_{p_1q_1}$ and $\ell_{p_2q_2}$.}%
\label{lemfig}%
\end{figure}

We are therefore left only with the case where 
each of $\lambda\cap \ell_{p_1q_1}$ and $\lambda\cap \ell_{p_2q_2}$ is either empty or infinite. That is, we have $\ell_{p_1q_1}\parallel\ell_{p_2q_2}$ (since both are parallel to $\lambda$).
As has already been argued, we also have $\ell_{p_1p_2}\parallel\ell_{q_1q_2}$, 
and thus $p_1q_1q_2p_2$ is a parallelogram; see Figure~\ref{lemfig}(b).
In particular, $|p_1p_2|=|q_1q_2|$.
Let $\xi$ be the intersection point of $\ell_{p_1p_2}$ with $\lambda$, and let $\eta$ be the point such that $(\xi,\eta)$ satisfies \eqref{eqxy} with respect 
to both surfaces $\sigma_{p_1q_1}$, $\sigma_{p_2q_2}$. 
By construction $\ell_{p_1\xi}\parallel\ell_{q_1\eta}$ and $\ell_{p_2\xi}(=\ell_{p_1\xi})\parallel\ell_{q_2\eta}$.
Hence $\eta$ must lie on $\ell_{q_1q_2}$. It is now easily checked that the only way in which $(\xi,\eta)$ can lie on both surfaces 
$\sigma_{p_1q_1}$ and $\sigma_{p_2q_2}$ is when $p_1,q_1,p_2,q_2$ are all collinear; see Figure~\ref{lemfig}(b).

To recap, so far we have shown that for $p_1$, $q_1$, $p_2$, and $q_2$ as above, either
$|\sigma_{p_1q_1}\cap\sigma_{p_2q_2}|\le 3$, or $p_1$, $q_1$, $p_2$, and $q_2$ 
are collinear with $|p_1p_2|=|q_1q_2|$.
It can then be shown that, in the latter case, any point 
$(u,v)\in \sigma_{p_1q_1}\cap\sigma_{p_2q_2}$ must satisfy $u, v\in \ell_{p_1q_1}$; see Figure~\ref{figs}(b).
Thus, for a point $\pi\in\R^4$ incident to each of  $\sigma_{p_1q_1}$, $\sigma_{p_2q_2}$, neither of
$(\pi,\sigma_{p_1q_1})$, $(\pi,\sigma_{p_2q_2})$ is in $\mathcal{I}$.
In other words, $\sigma_{p_1q_1}\cap\sigma_{p_2q_2}\cap\mathcal I=\emptyset$ in this case.
This contradiction completes the proof of Proposition~\ref{prop:admis}.  $\hfill\qed$

%
%
%
%
%
%


\section{Unit-area triangles spanned by points on three lines}\label{sec:exa}


In this section we consider the special case where $S$ is contained in the union of three distinct lines $l_1$, $l_2$, $l_3$. 
More precisely, we write $S=S_1\cup S_2\cup S_3$, with $S_i\subset l_i$, for $i=1,2,3$, and we are only interested in the number of unit-area triangles
spanned by triples of points in $S_1\times S_2\times S_3$.
It is easy to see that in this case the number of unit-area triangles of this kind is $O(n^2)$.
Indeed, for any pair of points $p,q\in S_1\times S_2$, the line $\ell_{pq}'$ intersects $l_3$ in at most one point, 
unless $\ell_{pq}'$ coincides with $l_3$. Ignoring situation of the latter kind,
we get a total of $O(n^2)$ unit-area triangles. 
If no two lines among $l_1,l_2, l_3$ are parallel to one another, 
it can be checked that the number of pairs $(p,q)$ such that $\ell_{pq}'=l_3$ is at most a constant,
 thus contributing a total of at most $O(n)$ unit-area triangles. 
For the case where two (or more) lines among $l_1,l_2, l_3$ are parallel, 
the number of unit-area triangles is easily seen to be $O(n^2)$.

In this section we present a rather subtle construction that shows that this bound 
is tight in the worst case, for any triple of distinct lines. 
Instead of just presenting the construction, we spend some time showing its connection to 
a more general setup considered by Elekes and R\'onyai~\cite{ER00} 
(and also, in more generality, by Elekes and Szab\'o~\cite{ESz12}).

Specifically, the main result of this section is the following.
\begin{theorem}\label{main2}
For any triple of distinct lines $l_1, l_2,l_3$ in $\R^2$, and for any integer $n$, there exist subsets 
$S_1\subset l_1$, $S_2\subset l_2$, $S_3\subset l_3$, each of cardinality $\Theta(n)$, 
such that $S_1\times S_2\times S_3$ spans $\Theta(n^2)$ unit-area triangles. 
\end{theorem}

\noindent{\bf Proof.}
The upper bound has already been established (for any choice of $S_1,S_2,S_3$), so we focus on the lower bound.
We recall that by the area formula for triangles in the plane, if
\begin{equation}\label{area-for}
\frac12\left|
\begin{array}{ccc}
p_x& q_x& r_x\\
p_y& q_y& r_y\\
1&1&1
\end{array}\right|=1,
\end{equation}
then the points $p=(p_x,p_y)$, $q=(q_x,q_y)$ and $r=(r_x,r_y)$ form the 
vertices of a positively oriented unit-area triangle in $\R^2$. 
(Conversely, if $\Delta pqr$ has area 1 then the left-hand side of (\ref{area-for}) has value $\pm 1$, 
depending on the orientation of $(p,q,r)$..)

To establish the lower bound, we distinguish between three cases, depending on the number 
of pairs of parallel lines among $l_1,l_2,l_3$.

\paragraph{The three lines $l_1,l_2,l_3$ are mutually parallel.}

In this case we may assume without loss of generality that they are of the form
\begin{align*}
l_1&=\{(t,0)\mid t\in\R\},\\
l_2&=\{(t,1)\mid t\in\R\},\\
l_3&=\{(t,\alpha)\mid t\in\R\},
\end{align*}
for some $1<\alpha\in\R$. 
(We translate and rotate the coordinate frame so as to place $\ell_1$ at the $x$-axis and then apply an area-preserving 
linear transformation that scales the $x$- and $y$-axes by reciprocal values.)
We set
\begin{align*}
S_1&:=\{(x_i:=\tfrac{i}{1-\alpha},0)\mid i=1,\ldots,n\}\subset l_1,\\
S_2&:=\{(y_j:=\tfrac j\alpha,1)\mid  j=1,\ldots,n \}\subset l_2,\\
S_3&:=\{(z_{ij}:=i+j-2,\alpha)\mid i,j=1,\ldots,n\}\subset l_3.
\end{align*}

Clearly each of the sets $S_i$, $i=1,2,3$, is of cardinality $\Theta(n)$. 
Note that for every pair of indices $1\le i,j\le n$, we have
$
(1-\alpha)x_i+\alpha y_j- z_{ij}=2.
$
By (\ref{area-for}), every such pair $i,j$ corresponds to a unit-area triangle with vertices 
$(x_i,0)\in S_1$, $(y_j,1)\in S_2$ and $(z_{ij},\alpha)\in S_3$. That is, $S_1\times S_2\times S_3$ 
spans $\Omega(n^2)$ unit-area triangles.

\paragraph{There is exactly one pair of parallel lines among $l_1,l_2,l_3$.} 
Using an area-preserving affine transformation\footnote{In more generality than the transformation used in the first case, 
these are linear transformations with determinant $\pm 1$.} of $\R^2$ (and possibly re-indexing the lines), we may assume that
\begin{align*}
l_1&=\{(t,0)\mid t\in\R\},\\
l_2&=\{(t,1)\mid t\in\R\},\\
l_3&=\{(0,t)\mid t\in\R\}.
\end{align*}

We claim that in this case the sets
\begin{align*}
S_1&:=\{(x_i:=2^i+2,0)\mid i=1,\ldots,n\}\subset l_1,\\
S_2&:=\{(y_j:=2^j+2,1)\mid  j=1,\ldots,n \}\subset l_2,\\
S_3&:=\{(0,z_{ij}:=\tfrac1{1-2^{j-i}})\mid i,j=1,\ldots,n,\;\;i\neq j\}\subset l_3,
\end{align*}
span $\Omega(n^2)$ unit-area triangles. As before, $S_1$, $S_2$ and $S_3$ are each of cardinality $\Theta(n)$.

Using (\ref{area-for}), the triangle spanned by $(x_i,0)$, $(y_j,1)$, and $(0,z_{ij})$ has unit area if 
$$
\frac12\left|
\begin{array}{ccc}
x_i& y_j& 0\\
0& 1&z_{ij}\\
1&1&1
\end{array}\right|= 1,
$$
or
$$
\frac{x_i-z_{ij}(x_i-y_j)}2=1,
$$
or
$$
z_{ij}=\frac{x_i-2}{x_i-y_j}=\frac1{1-\frac{y_j-2}{x_i-2}}.
$$
Since the latter holds for every $1\le i\neq j\le n$, we get $\Omega(n^2)$ unit-area triangles, as claimed.

\paragraph{No pair of lines among $l_1,l_2,l_3$ are parallel.}
This is the most involved case.
Using an area-preserving affine transformation of $\R^2$ (that is, a linear map with determinant $\pm 1$ and a translation), we may assume that the lines are given by
\begin{align*}
l_1&=\{(t,0)\mid t\in\R\},\\
l_2&=\{(0,t)\mid t\in\R\},\\
l_3&=\{(t, -t+\alpha)\mid t\in\R\},
\end{align*}
for some $\alpha\in\R$.
By (\ref{area-for}) once again, 
the points $(x,0)\in l_1$, $(0,y)\in l_2$, and $(z,-z+\alpha)\in l_3$ span a unit-area triangle if
$$
\frac12\left|
\begin{array}{ccc}
x& 0& z\\
0& y& -z+\alpha\\
1&1&1
\end{array}\right|=1,
$$
or
$$
z=f(x,y):=\frac{xy-\alpha x- 2}{y-x}.
$$
Thus it suffices to find sets $X,Y,Z\subset\R$, each of cardinality $\Theta(n)$, such that
$$
\big|\{(x,y,z)\in X\times Y\times Z\mid z=f(x,y)\}\big|=\Omega(n^2);
$$
then the sets
\begin{align*}
S_1&:=\{(x,0)\mid x\in X\}\subset l_1,\\
S_2&:=\{(0,y)\mid y\in Y\}\subset l_2,\\
S_3&:=\{(z,-z+\alpha)\mid z\in Z\}\subset l_3,
\end{align*}
are such that $S_1\times S_2\times S_3$ spans $\Omega(n^2)$ unit-area triangles.

\paragraph{The construction of $S_1,S_2, S_3$: General context.}
As mentioned at the beginning of this section, rather than stating what $S_1,S_2,S_3$ are, we present  the machinery that
we have used for their construction, thereby demonstrating that this problem is a special case of the theory
of Elekes and R\'onyai~\cite{ER00}; we also refer the reader to the more recent related studies~\cite{ESz12,RSS1,RSdZ}.

One of the main results of Elekes and R\'onyai is the following.
(Note that the bound in (i) has recently been improved to $O(n^{11/6})$ in \cite{RSS1,RSdZ}.)
\begin{theorem}[{\bf Elekes and R\'onyai~\cite{ER00}}]\label{ER}
Let $f(x,y)$ be a bivariate real rational function. 
Then one of the following holds.\\
(i) For any triple of sets $A,B,C\subset \R$, each of size $n$,
$$
\big|\big\{(x,y,z)\in A\times B\times C\mid z=f(x,y)\big\}\big|=o(n^2).
$$
(ii) There exist univariate real rational functions $h,\varphi,\psi$, such that $f$ has one of the forms
\begin{align*}
f(x,y)&=h(\varphi(x)+\psi(y)),\\
f(x,y)&=h(\varphi(x)\psi(y)),\\
f(x,y)&=h\left(\tfrac{\varphi(x)+\psi(y)}{1-\varphi(x)\psi(y)}\right).
\end{align*}
\end{theorem}

Our problem is thus a special instance of the context in Theorem~\ref{ER}. 
Specifically, we claim that 
$\displaystyle{
f(x,y)=\frac{xy-\alpha x-2}{y-x}}
$ satisfies condition (ii) of the theorem, which in turn will lead to the (natural)
construction of the desired sets $S_1,S_2,S_3$ (see below for details).

So we set the task of describing a necessary and sufficient condition 
that a real bivariate (twice differentiable) function $F(x,y)$ is locally\footnote{Note that such
a local representation of $F$ allows one to construct sets $A,B,C$ showing that property (i) of Theorem~\ref{ER} does not hold for $F$, i.e., sets such that there are $\Theta(n^2)$ solutions of $z=F(x,y)$ in $A\times B\times C$. This, using Theorem~\ref{ER}, implies the validity of property (ii) (globally, and with rational functions).}  
of the form $F(x,y)=h(\vphi(x)+\psi(y))$, for 
suitable univariate twice differentiable functions $h,\varphi,\psi$ (not necessarily rational functions).
This condition is presented in~\cite{ER00} where its (rather straightforward) necessity is argued. 
It is mentioned in \cite{ER00} that the sufficiency of this test was observed by A. Jarai Jr.~(apparently 
in an unpublished communication). Since no proof is provided in~\cite{ER00}, 
we present here a proof, for the sake of completeness.
\begin{lemma}\label{lem:Jarai}
Let $F(x,y)$ be a bivariate twice-differentiable real function, and assume that neither of $F_x,F_y$ is identically zero. 
Let $D(F)\subset\R^2$ denote the domain of definition of $F$, and let $U$ be a connected component of the relatively
open set $D(F)\setminus \big(\{F_y=0\}\cup\{F_x=0\}\big)\subset \R^2$. 
We let
$
q(x,y):=F_x/F_y,
$
which is defined, with a constant sign, over $U$.
Then
\begin{equation}\label{test}
\frac{\partial^2(\log|q(x,y)|)}{\partial x\partial y}\equiv 0
\end{equation}
over $U$ if and only if $F$, restricted to $U$, is of the form
\begin{equation}\label{decomp}
F(x,y)= h(\vphi(x)+\psi(y)),
\end{equation}
for some (twice-differentiable) univariate real functions $\vphi$, $\psi$, and $h$. 
\end{lemma}
\noindent{\bf Proof.}
We show only sufficiency of the condition (\ref{test}), as its necessity can easily be verified (as argued in \cite{ER00}). 
Setting $g(x,y):=\ln|q(x,y)|$, equation (\ref{test}) becomes
$$
\frac{\partial^2g}{\partial x\partial y}\equiv 0,
$$
and then clearly $g$ must have the form
$$
g(x,y)=g_1(x)-g_2(y),
$$
for suitable differentiable univariate functions $g_1,g_2$. That is
$$
\ln|q(x,y)|=g_1(x)-g_2(y), 
$$
or 
$$
q(x,y)=\pm e^{g_1(x)}/e^{g_2(y)}.
$$
That is,
\begin{equation}\label{eqjarai}
F_x/F_y=\vphi'(x)/\psi'(y),
\end{equation}
where 
$\displaystyle \vphi(x):=\pm \int e^{g_1(x)}dx$ and $\displaystyle  \psi(y):=\int e^{g_2(y)}dy$ 
(the arbitrary constants in these indefinite integrals clearly do not matter). 
Note that $\varphi$ and $\psi$ are twice differentiable strictly monotone functions, and are thus injective.

Next, we express the function $F$ in terms of new coordinates $(u,v)$, given by
\begin{align*}u&=\vphi(x)+\psi(y),\\
v&=\vphi(x)-\psi(y),
\end{align*}
where $(u,v)$ range over the image of $U$ under this transformation; since $\varphi$ and $\psi$ are injections, the above system is invertible.
Then by the chain rule we have
\begin{align*}
F_x&=F_u u_x+F_v v_x=\vphi'(x)(F_u+F_v)\\
F_y&=F_uu_y+F_vv_y=\psi'(y)(F_u-F_v),
\end{align*}
or
$$
\frac{F_x}{\vphi'(x)}=F_u+F_v,\quad 
\frac{F_y}{\psi'(y)}=F_u-F_v,
$$
and thus
$$
\frac{F_x}{\vphi'(x)}-\frac{F_y}{\psi'(y)}\equiv2F_v.
$$
Using (\ref{eqjarai}), the last equation is
$
F_v\equiv 0.
$
This means that $F$ depends only on the variable $u$, so it has the form
$
F(x,y)=h(\vphi(x)+\psi(y)),
$
as claimed. \qed

\paragraph{The construction of $S_1, S_2, S_3$: Specifics.}

We next apply Lemma~\ref{lem:Jarai} to our specific function 
$f(x,y)=\frac{xy-\alpha x- 2}{y-x}$.
In what follows we fix a connected open set $U\subset D(f)\setminus\big(\{f_x=0\}\cup \{f_y=0\}\big)$, and restrict 
the analysis only to points $(x,y)\in U$.
We have
$$
f_x = \frac{y^2-\alpha y-2}{(y-x)^2} ,\quad\text{and}\quad f_y = \frac{-x^2+\alpha x+2}{(y-x)^2} .
$$
By assumption, the numerators are nonzero and of constant signs, and the denominator is nonzero, over $U$.
In particular, we have
\begin{equation} \label{fxfy}
\frac{f_x}{f_y} = \frac{(-x^2+\alpha x+2)^{-1}} {(y^2-\alpha y-2)^{-1}} .
\end{equation} 
That is, without explicitly testing that \eqref{test} holds, we see that $f_x/f_y$ 
has the form in \eqref{eqjarai}.
Hence Lemma~\ref{lem:Jarai} implies that $f(x,y)$ can be written as
$
\displaystyle f(x,y) = h(\varphi(x)+\psi(y)) ,
$
for suitable twice-differentiable univariate functions $\varphi$, $\psi$, and $h$,
where $\varphi$ and $\psi$ are given (up to additive constants) by
\begin{align*}
\varphi'(x)  &= -\frac{1}{x^2-\alpha x-2}, \\
\psi'(y)  &= \frac{1}{y^2-\alpha y-2} .
\end{align*}
As explained above, this already implies that $f$ satisfies property (ii) of Theorem~\ref{ER}. 

Straightforward integration of these expressions yields that, up to a common multiplicative factor, which can be dropped,
we have\footnote{Note also that $f$ is defined over $y\neq x$, whereas in our derivation we also had to exclude $\{f_x=0\}\cup\{f_y=0\}$,
i.e. $\{x=s_1\}\cup\{x=s_2\}\cup\{y=s_1\}\cup\{y=s_2\}.$ Nevertheless, the final expression coincides with $f$ also over these excluded lines.}
\begin{align*}
\varphi(x)  &= \ln \left| \frac{x-s_2}{x-s_1} \right|,\\
\psi(y)  &= \ln \left| \frac{y-s_1}{y-s_2} \right|,
\end{align*}
where $s_1$, $s_2$ are the two real roots of $s^2-\alpha s-2=0$.

We conclude that $
\displaystyle f(x,y) = \frac{xy-\alpha x-2}{y-x} 
$
is a function of
\begin{align*}
\varphi(x) + \psi(y) & = 
\ln \left| \frac{x-s_2}{x-s_1} \right| 
+ \ln \left| \frac{y-s_1}{y-s_2} \right| \\
 &=
\ln { \left| \frac{x-s_2}{x-s_1} \right| }\cdot
{ \left| \frac{y-s_1}{y-s_2}\right| } ,
\end{align*}
or, rather, a function of 
$\displaystyle
u = { \frac{x-s_2}{x-s_1} }
\cdot{ \frac{y-s_1}{y-s_2} } .
$
A tedious calculation, which we omit, 
shows that 
$$
f(x,y) = \frac{s_2-s_1u}{1-u} ,
$$
confirming that $f$ does indeed have one of the special forms in Theorem~\ref{ER} above.
That is,
$$
f(x,y)=h(\vphi(x)\psi(y)),
$$
where $h,\vphi,\psi$ are the rational functions 
$$
h(u)=\frac{s_2-s_1u}{1-u},\quad
\vphi(x)=\frac{x-s_2}{x-s_1},\quad
\psi(y)=\frac{y-s_1}{y-s_2}
$$
(these are not the $\vphi,\psi$ in the derivation above).

We then choose points $x_1,\ldots,x_n,y_1,\ldots, y_n\in\R$ such that
$$
\frac{x_i-s_2}{x_i-s_1}=\frac{y_i-s_2}{y_i-s_1}=2^i,
$$
or
$$ 
x_i=y_i=\frac{2^is_1-s_2}{2^i-1},
$$
for $i=1,\ldots,n$, and let $X:=\{x_1,\ldots,x_n\}$ and $Y:=\{y_1,\ldots, y_n\}$. 
For $x=x_i$, $y=y_j$, the corresponding value of $u$ is $2^{i-j}$.
Hence, setting 
$$
Z:=\big\{f(x_i,y_j)\mid 1\le i,j\le n\big\}
=\left\{\frac{s_2-s_1\cdot 2^{i-j}}{1-2^{i-j}}\mid 1\le i,j\le n\right\},
$$
which is clearly also of size $\Theta(n)$, completes the proof. $\hfill\qed$


\section{Unit-area triangles in convex grids}\label{sec:conv}

A set $X=\{x_1,\ldots,x_n\}$, with $x_1<x_2<\cdots<x_n$, of real numbers is said to be \emph{convex} if
$x_{i+1}-x_i>x_i-x_{i-1}$,
for every $i=2,\ldots,n-1$. See \cite{ENR,SS11} for more details and properties  of convex sets.

In this section we establish the following improvement of Theorem~\ref{main1} for convex grids.
\begin{theorem}\label{main3}
Let $S=A\times B$, where $A,B\subset \R$ are convex sets of size $n^{1/2}$ each. Then the number of unit-area triangles spanned by the points of $S$ is $O(n^{31/14})$. 
\end{theorem}

\noindent{\bf Proof.} 
With each point $p=(a, b,c)\in A^3$ we associate a plane $h(p)$ in $\R^3$, given by
$$
\frac12\left|
\begin{array}{ccc}
a& b& c\\
x& y& z\\
1&1&1
\end{array}\right|=1,$$
or equivalently by
\begin{equation}\label{det}
(c-b)x+(a-c)y+(b-a)z=2.
\end{equation}
We put 
$
H:=\{h(p)\;\mid\;p\in A^3\}.
$

A triangle with vertices $(a_1,x_1),(a_2,x_2),(a_3,x_3)$ has unit area if and only if the left-hand side of (\ref{det}) has absolute value 1, 
so for half of the permutations $(i_1,i_2,i_3)$ of $(1,2,3)$, we have $(x_{i_1},x_{i_2},x_{i_3})\in h(a_{i_1},a_{i_2},a_{i_3})$. 
In other words, the number of unit-area triangles is at most one third of the number of incidences between the points of $B^3$ and the planes of $H$. 
In addition to the usual problematic issue that arise in point-plane incidence problems, 
where many planes can pass through a line that contains many points (see, e.g., \cite{AS05}), 
we need to face here the issue that the planes of $H$ are in general not distinct, and may arise with large multiplicity. Denote by $w(h)$ the \emph{multiplicity} of a plane $h\in H$, that is, $w(h)$ is the number of points $p\in A^3$ for which $h(p)=h$. Observe that, for $p,p'\in A^3$,
\begin{equation}
h(p)\equiv h(p')~\text{ if and only if }~p'\in p+(1,1,1)\R.
\end{equation}

We can transport this notion to points of $A^3$, by defining the \emph{multiplicity} $w(p)$ of a point $p\in A^3$ by
$$
w(p):= \left|(p+(1,1,1)\R)\cap A^3\right|.
$$
Then we clearly have $w(h(p))=w(p)$ for each $p\in A^3$.
Similarly, for $q\in B^3$, we put, by a slight abuse of notation,
$$
w(q):= \left|(q+(1,1,1)\R)\cap B^3\right|,
$$
and refer to it as the \emph{multiplicity} of $q$. (Clearly, the points of $B^3$ are all distinct, 
but the notion of their 
``multiplicity" will become handy in one of the steps of the analysis --- see below.)

Fix a parameter $k\in\N$, whose specific value will be chosen later. We say that $h\in H$ (resp., $p\in A^3$, $q\in B^3$) is $k$-\emph{rich}, if its multiplicity is at least $k$; otherwise we say that it is $k$-\emph{poor}. For a unit-area triangle $T$, with vertices $(a,x), (b,y), (c,z)$, we say that $T$ is \emph{rich-rich} (resp., \emph{rich-poor}, \emph{poor-rich}, \emph{poor-poor}) if $(a,b,c)\in A^3$ is $k$-rich (resp., rich, poor, poor), and $(x,y,z)\in B^3$ is $k$-rich (resp., poor, rich, poor). 
(These notions depend on the parameter $k$, which is fixed throughout this section.)

Next, we show that our assumption that $A$ and $B$ are convex allows us to have some control on the multiplicity of the points and the planes, which we need for the proof. 

For two given subsets $X,Y\subset \R$, and for any $s\in\R$, denote by $\delta_{X,Y}(s)$ 
the number of representations of $s$ in the form $x-y$, with $x\in X$, $y\in Y$. 
The following lemma is taken from Schoen and Shkredov. 
\begin{lemma}[{\bf Schoen and Shkredov~\cite{SS11}}]\label{lem:conv}
Let $X,Y\subset\R$, with $X$ convex. Then, for any $\tau\ge 1$, we have
$$
\Big|\{s\in X-Y~\mid~\delta_{X,Y}(s)\ge\tau\}\Big|=O\left(\frac{|X||Y|^2}{\tau^3} \right).
$$
\end{lemma}

Lemma \ref{lem:conv} implies that the number of points $(a,b)\in A^2$, for which the line $(a,b)+(1,1)\R$ contains at least $k$ points of $A^2$, is $O({n^{3/2}}/{k^2})$. Indeed, the number of differences $s\in A-A$ with $\delta_{A,A}(s)\ge\tau$ is $O(n^{3/2}/\tau^3)$. 
Each difference $s$ determines, in a 1-1 manner,  a line in $\R^2$ with orientation $(1,1)$ that contains the $\delta_{A,A}(s)$ pairs $(a,b)\in A^2$ with $b-a=s$. 
Let $M_\tau$ (resp., $M_{\ge \tau}$) denote the number of differences $s\in A-A$
with $\delta_{A,A}(s)=\tau$ (resp., $\delta_{A,A}(s)\ge \tau$). Then the desired number of 
points is
$$
\sum_{\tau\ge k}\tau M_\tau
=kM_{\ge k}+\sum_{\tau>k}M_{\ge \tau}
=O(n^{3/2}/k^2)+\sum_{\tau>k}O(n^{3/2}/\tau^3)
=O(n^{3/2}/k^2).
$$

We next establish the following simple claim.
\begin{lemma}\label{clm}
The number of $k$-rich points in $A^3$ and in $B^3$ is $O({n^2}/{k^2})$.
\end{lemma}
\begin{proof}
Let $(a,b,c)\in A^3$ be $k$-rich. Then, by definition, the line $l:=(a,b,c)+(1,1,1)\R$ contains at least $k$ points of $A^3$. 
We consider the line $l':=(a,b)+(1,1)\R$, which is the (orthogonal) projection of $l$ onto the $xy$-plane, which we identify with $\R^2$. 
Note that the projection of the points of $l\cap A^3$ onto $\R^2$ is injective and its image is equal to $l'\cap A^2$.  In particular, $l'$ contains at least $k$ points of $A^2$. As just argued, the total number of such points in $A^2$ (lying on some line of the form $l'$, that contains at least $k$ points of $A^2$) is $O({n^{3/2}}/{k^2})$. Each such point is the projection of at most $n^{1/2}$ $k$-rich points of $A^3$ (this is the maximum number of lines of the form $(a,b,c)+(1,1,1)\R$ that project onto the same line $l'$). Thus, the number of $k$-rich points in $A^3$ is $O(\frac{n^{3/2}}{k^2}\cdot n^{1/2})=O({n^2}/{k^2})$. The same bound applies to the number of $k$-rich points in $B^3$, by a symmetric argument. 
\end{proof}
\noindent\emph{Remark.} The proof of Lemma~\ref{clm} shows, in particular, that the images of the sets of $k$-rich points of $A^3$ and of $B^3$, under the projection map onto the $xy$-plane, are of cardinality $O({n^{3/2}}/{k^2})$.

In what follows, we bound separately the number of unit-area  triangles that are rich-rich, poor-rich (and, symmetrically, rich-poor), and poor-poor.

\paragraph{Rich-rich triangles.}

Note that for $((a,b,c),(\xi,\eta))\in A^3\times B^2$, with $a\neq b$, there exists at most one point $\zeta\in B$ such that $T((a,\xi),(b,\eta),(c,\zeta))$ has unit area. Indeed, the point $(c,\zeta)$ must lie on a certain line $l((a,\xi),(b,\eta))$ parallel to $(a,\xi)-(b,\eta)$. This line intersects $x=c$ in exactly one point (because $a\neq b$), which determines the potential value of $\zeta$. Thus, since we are now concerned with the number of rich-rich triangles (and focusing at the moment on the case where $a\neq b$), it suffices to bound the number of such pairs $((a,b,c),(\xi,\eta))$, with $(a,b,c)\in A^3$ being rich, and $(\xi,\eta)\in B^2$ being the projection of a rich point of $B^3$, which is $O((n^2/k^2)\cdot (n^{3/2}/k^2))=O(n^{7/2}/k^4)$, using Lemma~\ref{clm} and the Remark following it.

It is easy to check that the number of unit-area triangles $T(p,q,r)$, where $p,q,r\in P$ and $p,q$ share the same abscissa (i.e., $A$-component), is $O(n^2)$. Indeed, there are $\Theta(n^{3/2})$ such pairs $(p,q)$, and for each of them there exist at most $n^{1/2}$ points $r\in P$, such that $T(p,q,r)$ has unit area (because the third vertex $r$ must lie on a certain line $l(p,q)$, which passes through at most this number of points of $P$); here we do not use the fact that we are interested only in rich-rich triangles.
We thus obtain the following lemma.
\begin{lemma}\label{rr}
The number of rich-rich triangles spanned by $P$ is $ O\left(\frac{n^{7/2}}{k^4}+n^2\right)$.
\end{lemma}

\paragraph{Poor-rich and rich-poor triangles.}
Without loss of generality, it suffices to consider only poor-rich triangles. 
Put 
$$
H_i:=\{h\in H\;\mid\;2^{i-1}\le w(h)< 2^i\},
$$
for $i=1,\ldots,\log k$, and
$$
S_{\ge k}:=\{q\in B^3\;\mid\;w(q)\ge k\}.
$$
That is, by definition, $\bigcup_iH_i$ is the collection of $k$-poor planes of $H$, and $S_{\ge k}$ is the set of $k$-rich points of $B^3$. 
Since each element of $H_i$ has multiplicity at least $2^{i-1}$, we have the trivial bound $|H_i|\le n^{3/2}/2^{i-1}$.

Consider the family of horizontal planes ${\cal F}:=\{\xi_z\}_{z\in B}$, where $\xi_{z_0}:=\{z=z_0\}$. Our strategy is to restrict $S_{\ge k}$ and $H_i$, for $i=1,\ldots,\log k$, to the planes $\xi\in{\cal F}$, and apply 
the Szemer\'edi--Trotter incidence bound (see Theorem~\ref{ST}) to the resulting collections of points and intersection lines, on each such $\xi$. 
Note that two distinct planes $h_1,h_2\in H$ restricted to $\xi$, become two distinct lines in $\xi$. Indeed, each plane of $H$ contains a line 
parallel to $(1,1,1)$, and two such planes, that additionally share a horizontal line within $\xi$, must be identical. Using the Remark 
following Lemma~\ref{clm}, we have that the number of rich points $(x,y,z_0)\in S_{\ge k}$, with $z_0$ fixed, is $O\left({n^{3/2}}/{k^2}\right)$; that is, $|S_{\ge k}\cap \xi_{z_0}|=O\left({n^{3/2}}/{k^2}\right)$ for every fixed $z_0$.

The number of incidences between the points of $S_{\ge k}$ and the poor planes of $H$, counted with multiplicity (of the planes) is at most
$$
\sum_{z\in B}\sum_{i=1}^{\log k}2^i\cdot {\cal{I}}(S_{\ge k}\cap \xi_z,H_{iz}),
$$
where $H_{iz}:=\{h\cap \xi_z\mid h\in H_i\}$. By Theorem~\ref{ST}, this is at most
\begin{align*}
\sum_{z\in B}\sum_{i=1}^{\log k}&
2^i\cdot O\left( \left(\frac{n^{3/2}}{k^2}\right)^{2/3}\left(\frac{n^{3/2}}{2^{i-1}}\right)^{2/3}+ \frac{n^{3/2}}{k^2}+\frac{n^{3/2}}{2^{i-1}}\right)\\
=& \sum_{z\in B}
O\left(
\frac{n^2}{k^{4/3}}\sum_{i=1}^{\log k}2^{i/3}+\frac{n^{3/2}}{k^2}\sum_{i=1}^{\log k}2^i+n^{3/2}\log k\right)\\
=&\sum_{z\in B}
O\left(\frac{n^2}{k}+\frac{n^{3/2}}{k}+n^{3/2}\log k\right) \\
=&
O\left( \frac{n^{5/2}}{k}+n^2\log k \right).
\end{align*}

This bounds the number of poor-rich triangles spanned by $P$. Clearly, using a symmetric argument, this bound also applies to the number of rich-poor triangles spanned by $P$. We thus obtain the following lemma.
\begin{lemma}\label{pr}
The number of poor-rich triangles and of rich-poor triangles spanned by $P$ is $\displaystyle O\left( \frac{n^{5/2}}{k}+n^2\log k \right)$.
\end{lemma}
\paragraph{Poor-poor triangles.}
Again we are going to use Theorem \ref{ST}.  For $i=1,\ldots,\log k$, put
$$
S_i:=\{q\in B^3\;\mid\;2^{i-1}\le w(q)< 2^i\},
$$
and let $S_i'$, $H_i'$ be the respective (orthogonal) projections of $S_i$, $H_i$ to the plane $\eta:=\{x+y+z=1\}$. 
Note that $H_i'$ is a collection of lines in $\eta$. 
Moreover, arguing as above, two distinct planes of $H_i$ project to two distinct lines of $H_i'$, and thus the multiplicity 
of the lines is the same as the multiplicity of the original planes of $H_i$. Similarly, a point $q\in S_i$ with multiplicity 
$t$ projects to a point $q'\in S_i'$ with multiplicity $t$ (by construction, there are exactly $t$ points of $S_i$ that project 
to $q'$). 
These observations allow us to use here too the trivial bounds $|S_i'|\le n^{3/2}/2^{i-1}$, $|H_i'|\le n^{3/2}/2^{i-1}$, 
for $i=1,\ldots,\log k$.

Applying Theorem~\ref{ST} to the collections $S_i',H_j'$ in $\eta$, for $i,j=1,\ldots,\log k$, taking under account the 
multiplicity of the points and of the lines in these collections, we obtain that the number of incidences between the poor 
points and the poor planes, counted with the appropriate multiplicity, is at most
\begin{align*}
\sum_{i,j=1}^{\log k}2^{i+j}\cdot  I(S_i',H_j')
&= \sum_{i,j=1}^{\log k}2^{i+j}\cdot
O\left(
\left(\frac{n^{3/2}}{2^{i-1}}\right)^{2/3}\left(\frac{n^{3/2}}{2^{j-1}}\right)^{2/3}+\frac{n^{3/2}}{2^{i-1}}+\frac{n^{3/2}}{2^{j-1}}
\right)\\
&=O\left(
n^2\sum_{i,j=1}^{\log k} 2^{(i+j)/3}+n^{3/2}\sum_{i,j=1}^k\left(2^i+2^j\right)
\right)\\
&=
O\left(
n^2k^{2/3}+n^{3/2}k\log k
\right).
\end{align*}
Thus, we obtain the following lemma.
\begin{lemma}\label{pp}
The number of poor-poor triangles spanned by $P$ is $O\left(n^2k^{2/3}+n^{3/2}k\log k \right)$.
\end{lemma}

In summary, the number of unit-area triangles spanned by $P$ is
\begin{equation}
O\left(\frac{n^{7/2}}{k^4}+\frac{n^{5/2}}{k}+n^2k^{2/3}+n^{3/2}k\log k\right).
\end{equation}
Setting $k=n^{9/28}$ makes this bound $O(n^{31/14})$, and Theorem \ref{main3} follows. $\hfill\qed$

\vspace{1cm}
\noindent{\bf Acknowledgment.} We are grateful to Frank de Zeeuw for several very helpful comments that simplified some parts of the analysis.

\end{document}